\documentclass[11pt]{amsart}
\usepackage{amsfonts}
\usepackage{amssymb}
\usepackage{amsmath}
\usepackage{color}

\newtheorem{thm}{Theorem}[section]
\newtheorem{cor}[thm]{Corollary}
\newtheorem{lem}[thm]{Lemma}

\newtheorem{propp}[thm]{Proposition}

\newenvironment{theorem}{\begin{thm}\setcounter{equation}{0}
}{\end{thm}}
\newenvironment{corollary}{\begin{cor}\setcounter{equation}{0}
}{\end{cor}}
\newenvironment{lemma}{\begin{lem}\setcounter{equation}{0}
}{\end{lem}}

\begin{document}
\title{On $n^{th}$ class preserving automorphisms of $n$-isoclinism family}
\author{Surjeet Kour}\thanks{}
\address{Discipline of Mathematics, Indian Institute of Technology, Gandhinagar  382355, India.}
\email{surjeetkour@iitgn.ac.in} 

\begin{abstract}
Let $G$ be a finite group and $M,N$ be two normal subgroups of $G$. Let $Aut_N^M(G)$ denote the group of all automorphisms of $G$ which fix $N$ element wise and act trivially on $G/M$. Let $n$ be a positive integer. In this article we have shown that if $G$ and $H$ are two $n$-isoclinic groups, then there exists an isomorphism from      
$Aut_{Z_n(G)}^{\gamma_{n+1}(G)}(G)$ to  $Aut_{Z_n(H)}^{\gamma_{n+1}(H)}(H)$, which maps the group of $n^{th}$ class preserving automorphisms of $G$ to the group of $n^{th}$ class preserving automorphisms of $H$. Also, for a nilpotent group of class at most $(n+1)$, with some suitable conditions on $\gamma_{n+1}(G)$, we prove that $Aut_{Z_n(G)}^{\gamma_{n+1}(G)}(G)$ is isomorphic to the group of inner automorphisms of a quotient group of $G$.
\end{abstract}

\maketitle

\section{Introduction} One of the most fundamental problem in group theory is the classification of groups and the notion of isomorphism between two groups has played a very important role in it. However, the notion of isomorphism which partitions the class of all groups into equivalence classes(isomorphism classes)is very strong. If one would like to have all abelian groups or all finite $p$-groups fall into one equivalence class then the notion of isomophism will not work. With these thoughts in mind, Hall in 1940 defined a more general equivalence relation in the class of all groups which could give a satisfactory classification of all finite $p$-groups and all abelian groups. In 1940 in \cite{Hall 1}, Hall introduced the notion of isoclinism, which is an equivalence relation in the class of all groups. The notion of isoclinism is weaker than isomorphism and all abelian groups form one equivalence class. Roughly speaking, two groups are isoclinic if their central quotients are isomorphic and their commutator subgroups are also isomorphic. 
More precisely, if $G$ and $H$ are two groups, we say, $G$ is isoclinic to $H$ if there exists an isomorphism $\alpha : G/Z(G) \rightarrow H/Z(H)$  and an isomorphism $\beta : \gamma_{2}(G) \rightarrow \gamma_{2}(H)$  such that the following diagram commutes.\\

\begin{center}
$
\begin{matrix}

 \frac{G}{Z(G)}\times  \frac{G}{Z(G)})          & \xrightarrow{\alpha \times \alpha}   & \frac{H}{Z(H)}\times  \frac{H}{Z(H)}\\                    
  & & \\
  \downarrow{\gamma_{G}}  &       & \downarrow{\gamma_{H}}\\
  &  & \\
 \gamma_{2}(G)        & \xrightarrow{\beta} & \gamma_{2}(H)              ,

\end{matrix}
$
\end{center}

where $Z(G)$ and $\gamma_2(G)$ denote the center of $G$ and the commutator subgroup of $G$ respectively. The maps $\gamma_G$ and $\gamma_H$ are given by $$\gamma_{G}(g_1Z(G), g_2Z(G)) = [g_1, g_2]$$ for all $g_1, g_2 \in G$ and   $$\gamma_{H}(h_1Z(H), h_2Z(H)) = [h_1, h_2]$$ for all $h_1,h_2 \in H$.

 Notion of isoclinism has played a very important role in the classification of finite $p$-groups. Isoclinism divides $p$-groups into various isoclinism families and it has been studied by many mathematicians (see \cite{Hall 1, Hall 2}). Many authors investigated various group theoretical properties which are invariant under the isoclinism. Later, Hall in \cite{Hall 2}, generalized the notion of isoclinism to that of isologism, which is in fact an isoclinism with respect to certain varieties of groups. 

Hekster in 1986 in \cite{Hekster}, conceptualized and studied the notion of $n$-isoclinism over the variety of all nilpotent groups of class at most n. In this article, he has also done an extensive study of group theoretical properties which are invariant under $n$-isoclinism. He has shown that most of the known results for isoclinism can carry over to $n$-isoclinism. In recent years, people also started study of various subgroups of the automorphism group, which are invariant(isomorphic) for isoclinic groups(see \cite{Manoj, Rai}).

Let $M$ and $N$ be two normal subgroups of $G$. Let $Aut^{M}(G)$ denote the group of all automorphisms of $G$ which fix $M$ set wise and act trivially on $G/M$ and let $Aut_{N}(G)$ denote the group of all automorphisms of $G$ which fix  $N$ element wise. The group $Aut^{M}(G)\cap Aut_{N}(G)$ is denoted by $Aut_{N}^{M}(G)$. Let $n$ be a positive integer. An automorphism $\theta$ of $G$ is called $n^{th}$ class preserving if for each $g\in G$, there exists $x\in \gamma_{n}(G)$ such that $\theta(g)= x^{-1}gx$,
where $\gamma_n(G)$ denotes the $n^{th}$ term of the lower central series. We denote the group of all $n^{th}$ class preserving automorphisms by $Aut_{c}^{n}(G)$. Note that $Aut_c^1(G)$ is denoted by $Aut_c(G)$ and called the group of all class preserving automorphisms.

In \cite{Manoj}, Yadav proved that if $G$ and $H$ are two finite non abelian isoclinic groups, then $Aut_c(G)$  is isomorphic to $Aut_c(H)$. In \cite{Rai}, Rai extended Yadav's result to the group $Aut_{Z(G)}^{\gamma_2(G)}(G)$. He proved that, if $G$ and $H$ are two finite non abelian isoclinic groups then there exists an isomorphism $ \phi : Aut_{Z(G)}^{\gamma_2(G)}(G)\rightarrow Aut_{Z(H)}^{\gamma_2(H)}(H)$ such that $ \phi(Aut_c(G))= Aut_c(H)$. In this article we study these subgroups of the automorphism group for an $n$-isoclinism family. More precisely, in Theorem \ref{t4}, 
 we prove that if $G$ and $H$ are two finite  $n$-isoclinic groups then there exists an isomorphism $\Psi : Aut_{Z_n(G)}^{\gamma_{n+1}(G)}(G) \rightarrow Aut_{Z_n(H)}^{\gamma_{n+1}(H)}(H)$ such that $\Psi(Aut_c^n(G)) = Aut_c^n(H)$. Rai \cite [Theorem A ]{Rai} , the extension of Yadav \cite[Theorem 4.1]{Manoj},  is obtained as Corollary \ref{c5} of Theorem \ref{t4}.

In the same article \cite{Rai}, Rai also proved that if $G$ is a finite $p$-group of nilpotency class two, then $ Aut_{Z(G)}^{\gamma_2(G)}(G) = Inn(G)$ if and only if $\gamma_2(G)$ is cyclic. In section 3, we generalize this result in the following two ways. 

\begin{enumerate}
\item We consider an arbitrary nilpotent group(not just a finite $p$-group).
\item We study $ Aut_{Z_n(G)}^{M}(G)$ for a central subgroup $M$ of $G$ which contains $\gamma_{n+1}(G)$.
\end{enumerate}

In Theorem \ref{t5},  we prove that if $G$ is a finite non abelian group with a central subgroup $M$ such that $\gamma_{n+1}(G) \subseteq M$, then
$$ Aut_{Z_n(G)}^M(G) \simeq Inn(G/Z_{n-1}(G))$$ if and only if $M_{p_i}$ is cyclic for each $p_i \in \pi(G/ Z_{n-1}(G))$, where $M_{p_i}$ denotes the $p_i$-primary component of $M$. We obtain Rai \cite[Theorem B(2)]{Rai} as Corollary \ref{c2} of Theorem \ref{t5}.

\section{{Notations and Preliminaries}}\label{s2}
In this section, we recall a few definitions and some known results. Let $G$ be a finite group. We denote the identity of a group by 1. Throughout the article n denotes an integer and p denotes a prime number. The lower central series of a group $G$ is the series; $ \gamma_1(G) \geq \gamma_2(G) \geq \cdots$ defined as $\gamma_1(G)=G$ and  $\gamma_{n+1}(G) = [\gamma_n(G), G]$ for all $n \geq 1$. Note  that each $\gamma_i(G)$ is a characteristic subgroup of $G$ and $\gamma_2(G)$ is called the commutator subgroup of $G$ . $G$ is a nilpotent group of class $n$ if and only if $\gamma_{n+1}(G)= \{1\}$ and $\gamma_{n}(G)\neq \{1\}$.

The upper central series of $G$ is a sequence of normal subgroups $\{1\}= Z_0(G) \leq Z_1(G) \leq Z_2(G) \leq \cdots$ such that $Z_i(G)/Z_{i-1}(G) = Z(G/Z_{i-1}(G))$. From the definition of upper central series it follows that $ g \in Z_n(G)$ if and only if $[g, g_1, \ldots, g_n]= 1$, where $g_i \in G$ for $1 \leq i \leq n$. Hence $[\gamma_n(G), Z_n(G)]=\{1\}$. With this observation, it is easy to see that there exists a natural well defined map \\
$$\gamma(n, G): \frac{G}{Z_n(G)} \times \cdots \times \frac{G}{Z_n(G)} \rightarrow \gamma_{n+1}(G) $$ given by 
$$\gamma(n, G)(g_1Z_n(G), \ldots ,g_{n+1}Z_n(G))= [g_1, g_2, \ldots , g_{n+1}].$$

Now we recall the notions of $n$-homoclinism and $n$-isoclinism which are introduced by Hall \cite{Hall 2} and extensively investigated by Hekster \cite{Hekster}.

Let $G$ and $H$ be two finite groups. A pair $(\alpha, \beta)$ is called $n$-homoclinism between $G$ and $H$, if $\alpha : G/Z_n(G) \rightarrow H/Z_n(H)$ is a group homomorphism and $\beta : \gamma_{n+1}(G) \rightarrow \gamma_{n+1}(H)$ is a group homomorphism such that the following diagram commutes\\

\begin{center}
$
\begin{matrix}

 \frac{G}{Z_n(G)}\times \cdots \times \frac{G}{Z_n(G)})          & \xrightarrow{\alpha^{n+1}}   & \frac{H}{Z_n(H)}\times \cdots \times \frac{H}{Z_n(H)}\\                    
  & & \\
  \downarrow{\gamma(n,G)}  &       & \downarrow{\gamma(n, H)}\\
  &  & \\
 \gamma_{n+1}(G)        & \xrightarrow{\beta} & \gamma_{n+1}(H).              

\end{matrix}
$
\end{center}
\vspace{4.5mm}

Observe that the notion of $n$-homoclinism generalizes the notion of homomorphism. In fact $0$-homoclinisms are nothing but homomorphisms. If $\alpha$ is surjective then $\beta$ is surjective and if $\beta$ is injective then $\alpha$ is injective. We say $G$ and $H$ are $n$-isoclinic if $\alpha$ and $\beta$ are isomorphisms. In this case, the pair $(\alpha, \beta)$ is called $n$-isoclinism between $G$ and $H$.

Note that $n$-isoclinism is an equivalence relation on the class of all groups and an equivalence class is called an $n$-isoclinism family. 

\vspace{4.5mm}

Now we recall some well known results on $n$-isoclinism($n$-homoclinism).

\begin{lemma}\cite[Lemma 3.8]{Hekster}\label{l1}
Let $(\alpha, \beta)$ be an $n$-homoclinism from $G$ to $H$ and let $x \in \gamma_{n+1}(G)$. Then the following statements hold:
\begin{enumerate}
\item $\alpha(xZ_n(G)) = \beta(x)Z_n(H)$.
\item For $g \in G$ and $h \in \alpha(gZ_n(G))$, $\beta(g^{-1}xg)=h^{-1}\beta(x)h$.
\end{enumerate}
\end{lemma}

We also recall the following theorem of Hekster.
\begin{theorem}\cite[Theorem 3.12]{Hekster}\label{t1}
Let $(\alpha, \beta)$ be an $n$-isoclinism from $G$ to $H$. Then for all $i \geq 0$,  $~~\beta(\gamma_{n+1}(G) \cap Z_i(G)) = \gamma_{n+1}(H) \cap Z_i(H)$.
\end{theorem}

The following lemma is an easy observation.
\begin{lemma}\label{l5}
Let $G$ be a group and $n$ be a positive integer. Then for all $n \geq 1$,  $ Aut_c^n(G) \subseteq Aut_{Z_n(G)}^{\gamma_{n+1}(G)}(G)$.
\end{lemma}
\begin{proof}
Let $ f \in Aut_c^n(G)$. Then for each $g \in G$ there exists $x \in \gamma_{n}(G)$ such that $ f(g) = x^{-1}gx$. Hence $g^{-1}f(g)= [g,x]\in \gamma_{n+1}(G)$. Furthermore, if $g \in Z_n(G)$ then $g^{-1}f(g) = [g,x]=1$. Thus $f(g)=g$ for all $g \in Z_n(G)$.
\end{proof}

The exponent of a group $G$ is the smallest positive integer $n$ such that $g^n =1$ for all $g \in G$. We denote exponent of a group $G$ by $exp(G)$. For a finite group $G$, $\pi(G)$ denotes the set of all prime divisors of order of $G$. If $G$ is finite abelian, then $ G = \prod_{p \in \pi(G)} G_p$, where $G_p$ denotes the $p$-primary component of $G$.

The following lemmas are well known.
\begin{lemma}\label{l3}
Let $G$ be a nilpotent group of class $(n+1)$. Then $exp(G/Z_n(G)) = exp(\gamma_{n+1}(G))$.
\end{lemma}

\begin{lemma}\label{l4}
Let $G$ and $H$ be two finite groups such that $exp
(H) \mid exp(G)$. Then $\pi(H) \subseteq \pi(G)$.
\end{lemma}

A subgroup $H$ of $G$ is called central if $H \subseteq Z(G)$. For an abelian group $H$, $Hom(G,H)$ denotes the abelian group of all homomorphisms from $G$ to $H$. Now we recall the following theorem from \cite{kour}.

\begin{theorem}\cite[Theorem 3.3]{kour}\label{t2}
Let $G$ be a group and let $M$ and $N$ be two normal subgroups of $G$. Suppose, $M$ is a central subgroup of $G$. Then the following statements are true:
\begin{enumerate}
\item For each $ f \in Aut_N^M(G)$, the map $\alpha_f : G/N \rightarrow M $ given by $\alpha_f(gN) = g^{-1}f(g)$ is well
defined and it is a homomorphism.

\item If $G$ is finite and $M \subseteq N$, then the map $\phi : Aut_N^M(G) \rightarrow Hom(G/N , M)$ defined by
$\phi(f) =  \alpha_f $ is an isomorphism.
\end{enumerate}
\end{theorem}

The following result is due to Azhdari and Akhavan-Malayeri \cite{Maly}.

\begin{lemma}\cite[Proposition 1.3]{Maly}\label{l5}
Let $G$ and $H$ be two finite abelian $p$-groups and let $exp(H)\mid exp(G)$. Then $Hom(G, H) \simeq G$ if and only if $H$ is cyclic.
\end{lemma}

\section{$n^{th}$-class preserving automorphisms of $n$-isoclinism family}\label{s3}

Throughout the section $n$ denotes a positive integer. In this section we study the subgroups of the automorphism group of $n$-isoclinic groups. In \cite{Rai}, Rai proved that, if $G$ and $H$ are two finite non abelian 1-isoclinic(isoclinic) groups then there exists an isomorphism $ \phi : Aut_{Z(G)}^{\gamma_2(G)}(G)\rightarrow Aut_{Z(H)}^{\gamma_2(H)}(H)$ such that $ \phi(Aut_c(G))= Aut_c(H)$. Here, we extend these results to $n$-isoclinic groups. Theorem A of Rai \cite{Rai} is a special case of Theorem \ref{t4}. If $G$ is a nilpotent group of class at most $(n+1)$. We also obtained a necessary and sufficient condition on $\gamma_{n+1}(G)$ such that $Aut_{Z_n(G)}^{\gamma_{n+1}(G)}(G)$ is isomorphic to the group of inner automorphisms of a quotient group of $G$. Rai \cite[Theorem B]{Rai} is obtained as a corollary to Theorem \ref{t5}.

Note that the subgroups of the automorphism group, which we study in this section, are trivial for an abelian group. Thus we may further assume that all groups are non-abelian.

\begin{lemma}\label{l2}
Let $G$ and $H$ be two finite groups and $(\alpha, \beta)$ be an $n$-isoclinism from $G$ to $H$. Then the following statements are true:
\begin{enumerate}
\item For each $f \in Aut_{Z_n(G)}^{\gamma_{n+1}(G)}(G)$, the map $\theta_f : H \rightarrow H$ given by $\theta_f(h)= h \beta(g^{-1}f(g))$, where $ g \in G $ such that $\alpha(gZ_n(G)) = hZ_n(H)$, is well defined.
\item For all $h \in H$, $h^{-1}\theta_f(h) \in \gamma_{n+1}(H)$.
\item If $f_1, f_2 \in Aut_{Z_n(G)}^{\gamma_{n+1}(G)}(G)$ such that $f_1 \neq f_2$, then $\theta_{f_1} \neq \theta_{f_2}$.
\end{enumerate} 
\end{lemma}

\begin{proof}
\begin{enumerate}
\item Suppose $g_1Z_n(G)= g_2Z_n(G)$ for some $g_1, g_2 \in G$. Then $g_1^{-1}g_2 \in Z_n(G)$. As $f$ fixes $Z_n(G)$ element wise, $f(g_1^{-1}g_2)=g_1^{-1}g_2$. This implies that $g_1^{-1}f(g_1) = g_2^{-1}f(g_2)$.
Hence $h \beta(g_1^{-1}f(g_1)) = h \beta(g_2^{-1}f(g_2))$.
\\
\item Trivial.
\\
\item Let $f_1, f_2 \in Aut_{Z_n(G)}^{\gamma_{n+1}(G)}(G)$ and let $f_1 \neq f_2$. Then there exists  $1 \neq g \in G$ such that $f_1(g) \neq f_2(g)$. Therefore, $g^{-1}f_1(g) \neq g^{-1}f_2(g)$. Since $\beta$ is injective, $\beta(g^{-1}f_1(g)) \neq \beta(g^{-1}f_2(g))$. Hence $\theta_{f_1}(h) \neq \theta_{f_2}(h)$, where $h \in H$ such that $\alpha(gZ_n(G)) = hZ_n(H)$.
\end{enumerate}
\end{proof}

\begin{theorem}\label{t3}
Let $G$ and $H$ be two finite groups and $(\alpha, \beta)$ be an $n$-isoclinism from $G$ to $H$. Then the following statements are true:
\begin{enumerate}
\item For each $f \in Aut_{Z_n(G)}^{\gamma_{n+1}(G)}(G)$, the map $\theta_f$ defined in Lemma \ref{l2}(1), is a group homomorphism.
\item $\theta_f$ is an isomorphism fixing $Z_n(H)$ element wise.
\end{enumerate}
\end{theorem}

\begin{proof}
\begin{enumerate}
\item From Lemma \ref{l2}, $\theta_f$ is  well defined  and $h^{-1}\theta_f(h) \in \gamma_{n+1}(H)$. Let $h_1, h_2 \in H$ and let $g_1, g_2 \in G$ such that $\alpha(g_1Z_n(G))= h_1Z_n(H)$ and $\alpha(g_2Z_n(G))= h_2Z_n(H)$.  Then $\alpha(g_1g_2Z_n(G))= h_1h_2Z_n(H)$ and 
 
$$
\begin{array}{lcl}
\theta_f(h_1h_2) &=&h_1h_2\beta( g_2^{-1}g_1^{-1}f(g_1g_2))
\\[2mm]
&=& h_1(h_2 \beta(g_2^{-1}g_1^{-1}f(g_1)g_2g_2^{-1} f(g_2)))
\\[2mm]
&=& h_1(h_2 \beta(g_2^{-1}g_1^{-1}f(g_1)g_2) h_2^{-1}) h_2 \beta(g_2^{-1}f(g_2)). 
\\[2mm]
\end{array}
$$

Note that $g_1^{-1}f(g_1) \in \gamma_{n+1}(G)$. Therefore by using Lemma \ref{l1}(2), $$\beta(g_2^{-1}g_1^{-1}f(g_1)g_2)= h_2^{-1}\beta(g_1^{-1}f(g_1))h_2.$$
Hence $\theta_f(h_1h_2)= h_1\beta(g_1^{-1}f(g_1)) h_2 \beta(g_2^{-1}f(g_2))= \theta_f(h_1)\theta_f(h_2)$. 
\\
\item It is enough to prove that $\theta_f$ is injective. Let $\theta_f(h)=1$ for some $h \in H$. Then $h \beta(g^{-1}f(g))= 1$, where $ g \in G$ such that $\alpha(gZ_n(G)) = hZ_n(H)$. Hence $h =(\beta(g^{-1}f(g)))^{-1} \in \gamma_{n+1}(H)$. Without loss of generality, we may assume that $ g \in \gamma_{n+1}(G)$. If not, then there exists $g' \in \gamma_{n+1}(G)$ such that $\beta(g')=h$ and $g'Z_n(G)= gZ_n(G)$. Hence we can replace $g$ by $g'$.

Now observe that, $ 1 = h \beta(g^{-1}f(g)) = h \beta(g^{-1}) \beta(f(g)) \in \gamma_{n+1}(H)$. Hence $ \beta(f(g)) =h^{-1}\beta(g) \in \gamma_{n+1}(H)$. Furthermore, $\alpha(gZ_n(G))= \beta(g)Z_n(H) = hZ_n(H)$, implies that, $h^{-1}\beta(g)\in Z_n(H)$. Therefore, 
$ \beta(f(g)) \in \gamma_{n+1}(H)\cap Z_n(H)$. By using Theorem \ref{t1}, we have $f(g) \in Z_n(G)\cap \gamma_{n+1}(G)$. As $f$ fixes $Z_n(G)$ element wise, $f(g)=g$ and $h =1$. Hence $\theta_f$ is an isomorphism. Also, if $h \in Z_n(H)$, then $g \in Z_n(G)$, where $g \in G$ such that $\alpha(gZ_n(G)) = hZ_n(H)$. Therefore, $\theta_f(h)=h$. Hence  $\theta_f \in Aut_{Z_n(H)}^{\gamma_{n+1}(H)}(H)$.

\end{enumerate}
\end{proof}

\begin{theorem}\label{t4}
Let $G$ and $H$ be two finite groups and let $(\alpha, \beta)$ be an $n$-isoclinism from $G$ to $H$. Then there exists an isomorphism $\Psi : Aut_{Z_n(G)}^{\gamma_{n+1}(G)}(G) \rightarrow Aut_{Z_n(H)}^{\gamma_{n+1}(H)}(H)$ such that $\Psi(Aut_c^n(G)) = Aut_c^n(H)$.
\end{theorem}

\begin{proof}
Define $\Psi :  Aut_{Z_n(G)}^{\gamma_{n+1}(G)}(G) \rightarrow Aut_{Z_n(H)}^{\gamma_{n+1}(H)}(H)$ such that $\Psi(f) = \theta_f$, where $\theta_f$ is the map defined in Lemma \ref{l2}(1). By Theorem \ref{t3},  $\theta_f \in Aut_{Z_n(H)}^{\gamma_{n+1}(H)}(H)$. 

First we show that $\Psi$ is a group homomorphism. Let $f_1, f_2 \in Aut_{Z_n(G)}^{\gamma_{n+1}(G)}(G)$. Then we need to show that 
$\theta_{f_1 \circ f_2}= \theta_{f_1}\circ 
\theta_{f_2}$. Consider $h \in H$, then
$$(\theta_{f_1}\circ \theta_{f_2})(h)=\theta_{f_1}(h \beta(g^{-1}f_2(g))),$$
where $g \in G$ such that $\alpha(g Z_n(G)) = hZ_n(H)$.\\
Note that 
$$
\alpha(f_2(g)Z_n(G))= \alpha(gZ_n(G))\alpha(g^{-1}f_2(g)Z_n(G))= hZ_n(H)\alpha(g^{-1}f_2(g)Z_n(G)).
$$

Since $g^{-1}f_2(g)\in \gamma_{n+1}(G)$, by Lemma \ref{l1} we have $$\alpha(g^{-1}f_2(g)Z_n(G))= \beta( g^{-1}f_2(g))Z_n(H)$$
Therefore, $\alpha(f_2(g)Z_n(G))= h\beta( g^{-1}f_2(g))Z_n(H)$. Thus

$$
\begin{array}{lcl}
(\theta_{f_1}\circ \theta_{f_2})(h)&=&\theta_{f_1}(h \beta(g^{-1}f_2(g)))
\\[2mm]
&=& h \beta(g^{-1}f_2(g)) \beta(f_2(g^{-1})f_1(f_2(g)) 
\\[2mm]
&=& h \beta(g^{-1} (f_1 \circ f_2)(g))
\\[2mm]
&=&\theta_{f_1 \circ f_2}(h).
\\[2mm]
\end{array}
$$ 
Hence $\Psi$ is a group homomorphism. 

In order to prove that $\Psi$ is an isomorphism, define a map $\Phi$ from $Aut_{Z_n(H)}^{\gamma_{n+1}(H)}(H)$ to $Aut_{Z_n(G)}^{\gamma_{n+1}(G)}(G)$ such that $\Phi(\theta) = f_{\theta}$, where $f_{\theta} : G \rightarrow G$ is given by $f_{\theta}(g) = g \beta^{-1}(h^{-1}\theta(h))$, where $h \in H$ such that $\alpha^{-1}(hZ_n(H))= g Z_n(G)$. Using the fact that $\alpha$ and  $\beta$ are isomorphisms, one can prove that $\Phi$ is a group homomorphism and $\Phi(\theta_f) = f$, for $f \in Aut_{Z_n(G)}^{\gamma_{n+1}(G)}(G)$. Therefore, $(\Phi \circ \Psi)(f) =\Phi(\theta_f) = f$. Hence $\Phi \circ \Psi =I$. Similarly, one can show that $\Psi \circ \Phi = I$. Thus $\Psi$ is an isomorphism.

Next we show that $\Psi(Aut_c^n(G)) = Aut_c^n(H)$. Consider $ f \in Aut_c^n(G) \subseteq Aut_{Z_n(G)}^{\gamma_{n+1}(G)}(G)$. Clearly, $\theta_f \in Aut_{Z_n(H)}^{\gamma_{n+1}(H)}(H)$ and $ \theta_f(h)= h \beta(g^{-1}f(g))$ where $g\in G$ such that $\alpha(gZ_n(G))= hZ_n(H)$. Also, for  $g \in G$, there exists  $x \in \gamma_n(G)$ such that $f(g)= x^{-1}gx$. Therefore, $ \theta_f(h)= h \beta(g^{-1}x^{-1}gx)= h \beta[g, x]$.
Observe that, $\beta[g, x] = [h, y]$,  where $ y \in \gamma_n(H)$ such that $\alpha(xZ_n(G))= yZ_n(H)$. Hence $\theta_f(h)= y^{-1}hy$ and 
$\Psi(Aut_c^n(G)) \subseteq Aut_c^n(H)$. Similarly, we can show that $\Phi (Aut_c^n(H)) \subseteq Aut_c^n(G)$. Therefore, we have $\Psi(Aut_c^n(G)) = Aut_c^n(H)$. 
\end{proof}

\begin{corollary}\cite[Theorem A]{Rai}\label{c5}
Let $G$ and $H$ be two finite isoclinic groups. Then there exists an isomorphism $\Psi : Aut_{Z(G)}^{\gamma_{2}(G)}(G) \rightarrow Aut_{Z(H)}^{\gamma_{2}(H)}(H)$ such that $\Psi(Aut_c(G)) = Aut_c(H)$.
\end{corollary}

\begin{proof}
Put $n=1$ in Theorem \ref{t4}.
\end{proof}

Let $G$ be a finite non abelian group and let $M$ be a central subgroup of $G$. Let $M_p$ denote the $p$-primary component of $M$, where $p$ is a prime divisor of the order of $M$. In \cite{Rai}, Rai proved that, if $G$ is a finite $p$-group of nilpotency class two then  $Aut_{Z(G)}^{\gamma_{2}(G)}(G)= Inn(G)$ if and only if $\gamma_2(G)$ is cyclic.
We generalize this result to a nilpotent group of class at most $(n+1)$. Also we have shown that the result is true for any central subgroup $M$ of $G$ which contains $\gamma_{n+1}(G)$. 

\begin{theorem}\label{t5}
Let $G$ be a finite non abelian group and let $M$ be a central subgroup of $G$ such that $\gamma_{n+1}(G) \subseteq M$. Then 
$$ Aut_{Z_n(G)}^M(G) \simeq Inn(G/Z_{n-1}(G))$$ if and only if $M_{p_i}$ is cyclic for each $p_i \in \pi(G/ Z_{n-1}(G))$.
\end{theorem}

\begin{proof}
Since $\gamma_{n+1}(G) \subseteq M \subseteq Z(G)$, $G$ is nilpotent of class at most $(n+1)$. If nilpotency class of $G$ is less than $(n+1)$, then both groups are trivial and result is true. Now we may assume that $G$ is nilpotent of class $(n+1)$. Therefore by Lemma \ref{l3}, $exp(G/Z_n(G))= exp(\gamma_{n+1}(G)) \mid exp(M)$. Hence, by using Lemma \ref{l4}, we have $\pi(G/Z_n(G)) \subseteq \pi(M)$. \\
Let $\pi(G/Z_n(G))=\{p_1, \ldots, p_r\}$ and let
$$ G/Z_n(G) \simeq H_{p_1} \times H_{p_2}\times \cdots \times H_{p_r},$$
where $H_{p_i}$ denotes the $p_i$-primary component of $G/Z_n(G)$.\\
Similarly, let $\pi(M) = \{p_1, \ldots, p_r, q_1, \ldots, q_s\}$ and 
$$ M \simeq M_{p_1} \times \cdots \times M_{p_r}\times M_{q_1} \times \cdots \times M_{q_s},$$
where $M_{p_i}$ and  $M_{q_i}$ denote the $p_i$-primary and $q_i$-primary components of $M$ respectively.\\
Clearly $exp(H_{p_i}) \mid exp(M_{p_i})$ for all $ 1 \leq i \leq r$. Now by Theorem \ref{t2}, we have 

$$
\begin{array}{lcl}
Aut_{Z_n(G)}^M(G)&\simeq & Hom (G/Z_n(G), M)
\\[2mm]
&= & Hom(\prod_{i=1}^r H_{p_i}, \prod_{i=1}^rM_{p_i} \times \prod_{j=1}^{s}M_{q_j})
\\[2mm]
&\simeq& Hom(\prod_{i=1}^r H_{p_i}, \prod_{i=1}^rM_{p_i})
\\[2mm]
&\simeq& \prod_{i=1}^r Hom(H_{p_i}, M_{p_i}).
\\[2mm]
\end{array}
$$ 

By Lemma \ref{l5}, $Hom(H_{p_i}, M_{p_i}) \simeq H_{p_i}$ if and only if $M_{p_i}$ is cyclic. Hence 
$$ Aut_{Z_n(G)}^M(G))\simeq \prod_{i=1}^r H_{p_i}\simeq G/Z_n(G)\simeq Inn (G/Z_{n-1}(G))$$ if and only if $M_{p_i}$ is cyclic for all $1 \leq i \leq r$.
\end{proof}

\begin{corollary}\label{c1}
Let $G$ be a finite non abelian group of class $(n+1)$. Then $ Aut_{Z_n(G)}^{\gamma_{n+1}(G)}(G) \simeq Inn(G/Z_{n-1}(G))$ if and only if $(\gamma_{n+1}(G))_{p}$ is cyclic for all $p$-primary components of $\gamma_{n+1}(G)$.
\end{corollary}
\begin{proof}
Take $M = \gamma_{n+1}(G)$.
\end{proof}

\begin{corollary}\label{c2}
Let $G$ be a finite non abelian $p$-group of class $(n+1)$. Then $ Aut_{Z_n(G)}^{\gamma_{n+1}(G)}(G) \simeq Inn(G/Z_{n-1}(G))$ if and only if $\gamma_{n+1}(G)$ is cyclic.
\end{corollary}

Rai's \cite[Theorem B(2)]{Rai} follows from corollary \ref{c2} by taking $n =1$. 

\vspace{0.7cm} \noindent {\bf Acknowledgement}

This research is supported by SERB-DST grant YSS/2015/001567.

\end{document}